\documentclass[12pt]{amsart}
\usepackage{amssymb}
\usepackage{amsmath}
\numberwithin{equation}{section}
\usepackage{graphicx}

\newtheorem{thm}{Theorem}[section]
\newtheorem{lemma}[thm]{Lemma}
\newtheorem{cor}[thm]{Corollary}
\newtheorem{prop}[thm]{Proposition}

\theoremstyle{definition}

\theoremstyle{remark}
\newtheorem{remark}[thm]{Remark}

\renewcommand{\S}{\mathfrak S}
\newcommand{\s}{\sigma}

\newcommand\Z{{\mathbb{Z}}}

\newcommand\PP{{\mathbb{P}}}
\newcommand\N{{\mathbb{N}}}

\newcommand\wh{\widehat}

\newcommand\bars{{\rm bar}}

\newcommand\NC{{\rm NC}}

\newcommand\bq{\begin{equation}}
\newcommand\eq{\end{equation}}
\newcommand\beq{\begin{eqnarray*}}
\newcommand\eeq{\end{eqnarray*}}
\newcommand\ben{\begin{enumerate}}
\newcommand\een{\end{enumerate}}
\newcommand\bit{\begin{itemize}}
\newcommand\eit{\end{itemize}}

\newcommand\des{{\rm des}}
\newcommand\asc{{\rm asc}}
\newcommand\exc{{\rm exc}}
\newcommand\inv{{\rm inv}}
\newcommand\maj{{\rm maj}}

\newcommand\ai{{\rm ai}}
\newcommand\aid{{\rm aid}}
\newcommand\sg{{\mathfrak S}}
\newcommand\Des{{\rm DES}}

\newcommand\Exc{{\rm EXC}}

\newcommand\Cov{{\rm Cov}}

\newcommand\NDA{{\rm NDA}}
\newcommand\NDD{{\rm NDD}}
\newcommand\cha{{\rm cocharge}}
\newcommand\pc{{\rm PC}}
\newcommand\pf{{\rm PF}}
\newcommand\wcomp{{\rm WComp}}

\def\wh{\widehat}
\def\hz{\hat 0}

\def\rh{\tilde{H}}
\def\rk{{\sf r}{\sf k}}

\def\zz{{\mathbb Z}}
\def\nn{{\mathbb N}}
\def\ff{{\mathbb F}}

\def\pp{{\mathbb P}}

\def\dd{{\mathcal D}}

\def\nn{{\mathbb N}}
\def\psx{{\mathcal P}_{\sf stab}}

\def\hz{\hat{0}}

\begin{document}

\title[Rees products]
{Rees products and lexicographic shellability}

\author[Linusson] {Svante Linusson$^1$}
\address{Department of Mathematics, KTH-Royal Institute of Technology, SE-100 44, Stockholm, Sweden}
\thanks{$^1$ Linusson is a Royal Swedish Academy of Sciences Research Fellow supported by a grant from the Knut and Alice Wallenberg Foundation.}
\email{linusson@math.kth.se}

\author[Shareshian]{John Shareshian$^2$}
\address{Department of Mathematics, Washington University, St. Louis, MO 63130}
\thanks{$^{2}$Supported in part by NSF Grants
DMS 0604233 and 0902142}
\email{shareshi@math.wustl.edu}

\author[Wachs]{Michelle L. Wachs$^3$}
\address{Department of Mathematics, University of Miami, Coral Gables, FL 33124}
\email{wachs@math.miami.edu}
\thanks{$^{3}$Supported in part by NSF Grants
DMS 0604562 and 0902323}


\date{January 30, 2012}

\dedicatory{Dedicated to Adriano Garsia}

\maketitle

\begin{abstract} We use the theory of lexicographic shellability to provide various
examples in which the rank of the homology of a Rees product of two
partially ordered sets enumerates some set of combinatorial objects,
perhaps according to some natural statistic on the set.  Many of these
examples generalize a result of J. Jonsson, which says that the rank of
the unique nontrivial homology group of the Rees product of a truncated
Boolean algebra of degree $n$ and a chain of length $n-1$ is the number of
derangements in $\S_n$.\end{abstract}

\vbox{
\tableofcontents
}

\section{Introduction} \label{secintro}
Rees products of posets were defined and studied by A. Bj\"orner and V. Welker in \cite{bw}.  While the main results in \cite{bw} provide combinatorial analogues of constructions in commutative algebra, it has turned out that Rees products of certain posets are connected with permutation enumeration and permutation statistics.  The first indication of this connection is provided by a conjecture in \cite{bw}, which says that the reduced Euler characteristic of the order complex of the Rees product of the truncated Boolean algebra $B_n \setminus \{\emptyset\}$ and a chain of length $n-1$ is the number of derangements in the symmetric group $\mathfrak S_n$.  This conjecture was proved by J. Jonsson in \cite{jo}.   

As we shall describe below, generalizations of Jonsson's result, along with similar results have been proved.  Our purposes in this paper are 
\begin{enumerate}
\item to give additional examples of Rees products whose order complexes have reduced Euler characteristics that enumerate certain classes of combinatorial objects, possibly according to some natural statistic,  and 
\item to show how the theory of lexicographic shellability applies to certain Rees products, in particular relating the homology of the order complex of the Rees product of a lexicographically shellable poset $P$ with a poset whose Hasse diagram is a rooted $t$-ary tree to the homology of the order complexes of some rank-selected subposets of $P$.
\end{enumerate}
These two purposes are in fact intertwined.  We prove all of our results on reduced Euler characteristics of order complexes of Rees products using lexicographic shellings.

All posets studied in this paper are finite.  We call a poset $P$ {\it semipure} if for each $x \in P$, the lower order ideal $P_{\leq x}:=\{y \in P:y \leq x\}$ is pure, that is, any two maximal chains in $P_{\leq x}$ have the same length.  The {\it rank} $r_P(x)$ of such an element $x$ is the length of a maximal chain in $P_{\leq x}$.  Given semipure posets $P,Q$ with respective rank functions $r_P,r_Q$, the {\it Rees product} $P \ast Q$ is the poset whose underlying set is
\[
\{(p,q) \in P \times Q:r_P(p) \geq r_Q(q)\},
\]
with order relation given by $(p_1,q_1) \leq (p_2,q_2)$ if and only if all of the conditions
\begin{itemize}
\item $p_1 \leq_P p_2$,
\item $q_1 \leq_Q q_2$, and
\item $r_P(p_1)-r_P(p_2) \geq r_Q(q_1)-r_Q(q_2)$
\end{itemize}
hold.  In other words, $(p_2,q_2)$ covers $(p_1,q_1)$ in $P \ast Q$ if and only if
\begin{enumerate}
\item $p_2$ covers $p_1$ in $P$, and
\item either $q_2=q_1$ or $q_2$ covers $q_1$ in $Q$.
\end{enumerate}

In Figure 1,  the Rees product of the truncated Boolean algebra $B_3 \setminus \{\emptyset\}$ and the chain $C_{2}:=\{0<1<2\}$ is given. The element $(S,j)$ is written as $S^j$ with the set brackets and commas omitted.

\vspace{.2in}\begin{center}\includegraphics[width=3.5in]{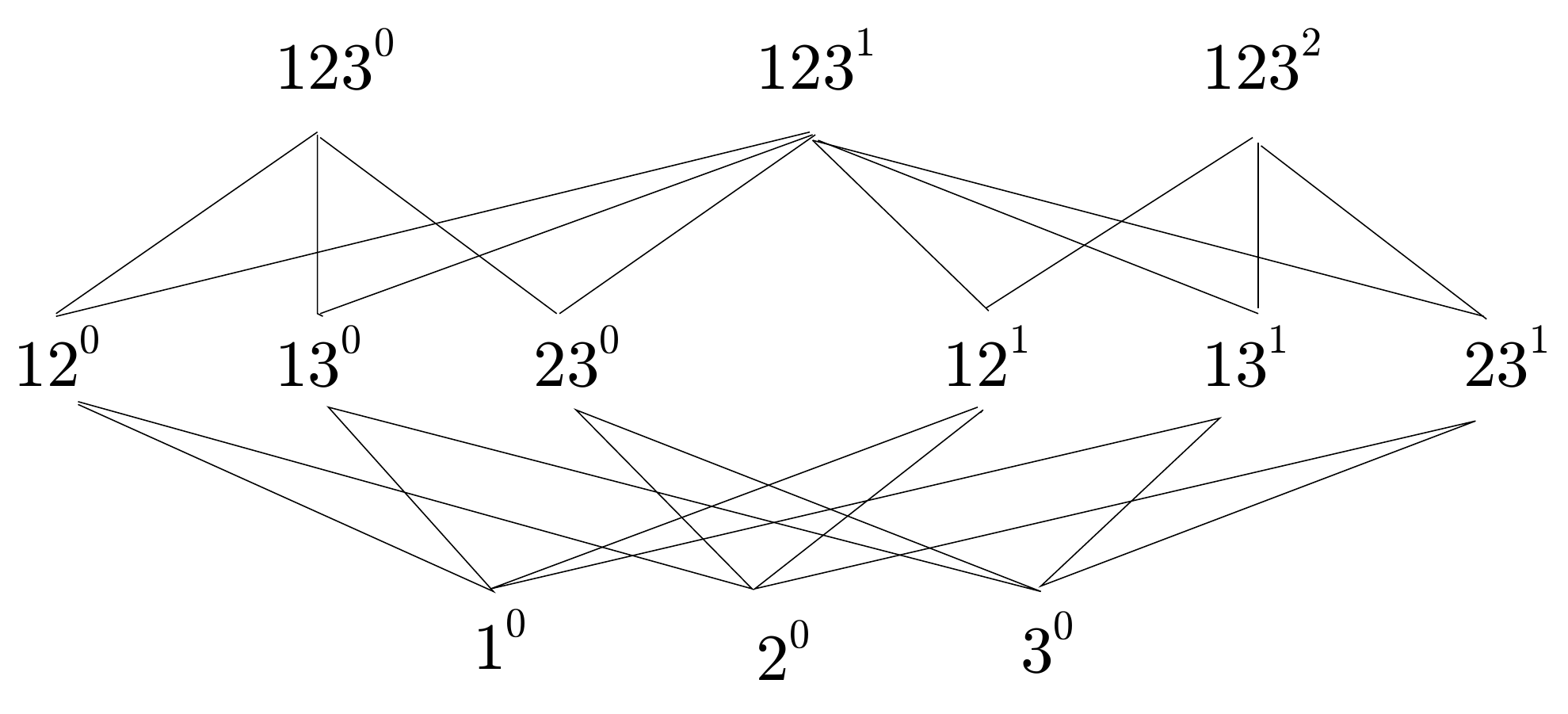}\end{center}
\centerline{{\bf Figure 1.}  $(B_3\setminus\{\emptyset\} ) * C_2$}
\vspace{.2in}

For any poset $P$, the {\it order complex} $\Delta P$ is the abstract simplicial complex whose $k$-dimensional faces are chains (totally ordered subsets) of length $k$ from $P$.  A simplicial complex $\Delta$ is {\it Cohen-Macaulay} if for each face $F \in \Delta$ (including the empty face), the reduced (integral, simplicial) homology of the link $lk_\Delta(F)$ is trivial in all dimensions except possibly $\dim(lk_\Delta(F))$.   Every Cohen-Macaulay complex is pure, that is, all maximal faces of a Cohen-Macaulay complex have the same dimension.   A poset is said to   be Cohen-Macaulay if its order complex is Cohen-Macaulay.   We will say that a poset has a particular topological property if its order complex has that property. 
The (reduced) homology of $P$ is given by $\rh_k(P):= \rh_k(\Delta P;\Z)$.    For further information on Cohen-Macaulay posets, see the surveys given  in  \cite{BGS}, \cite{st1}, \cite{w1}.

   Bj\"orner and Welker \cite[Corollary 2]{bw} prove that  the Rees product of any Cohen-Macaulay poset with any  acyclic Cohen-Macaualy poset   is Cohen-Macaulay.  It is known that both $B_n^-:=B_n \setminus \{\emptyset\}$ and the chain $C_n$ of length $n$ are Cohen-Macaulay, and $C_n$ is acyclic.  Thus the result of Jonsson mentioned above says that, with $d_n$ denoting the number of derangements in $\mathfrak S_n$,
\begin{equation} \label{jon}
\rk\widetilde{H}_{n-1}(B_n^- \ast C_{n-1})=d_n.
\end{equation}

Generalizations of (\ref{jon}) appear in the paper \cite{sw2} of Shareshian and Wachs.  For a poset $P$ with unique minimum element $\hz$, $P^-$ will denote $P \setminus \{\hz\}$.  For a prime power $q>1$ and a positive integer $n$, the poset of all subspaces of an $n$-dimensional vector space over the $q$-element field $\ff_q$ will be denoted by $B_n(q)$.  Also, $\dd_n$ will denote the set of all derangements in $\mathfrak S_n$.  It is shown in \cite{sw2} that
\begin{equation} \label{qan}
\rk\rh_{n-1}(B_n(q)^- \ast C_{n-1})=\sum_{\sigma \in \dd_n}q^{{{n} \choose {2}}-\maj(\sigma)+\exc(\sigma)},
\end{equation}
where $\maj$ and $\exc$ are, respectively, the major index and the excedance number, introduced by MacMahon in \cite[Vol. I, pp. 135,186; Vol. 2, p. viii]{mac1}, \cite{mac2} in the early part of the 20th century and extensively studied thereafter.

A generalization of (\ref{qan}) appears in \cite{sw2}.  
For positive integers $t,n$, let $T_{t,n}$ be the poset whose Hasse diagram is a complete $t$-ary tree of height $n$ with root at the bottom.   To put it more formally, $T_{t,n}$ consists of all sequences of elements of $[t]:=\{1,\ldots,t\}$ that have length at most $n$, including the empty sequence.  Given two such sequences ${\sf a}=(a_1,\ldots,a_k)$ and and ${\sf b} =(b_1,\ldots,b_l)$, we declare that ${\sf a} \leq {\sf b}$ if $k \leq l$ and $a_i=b_i$ for al $i \in [k]$.
Note that $T_{1,n}=C_n$.

It is shown in \cite{sw2} that if $P$ is Cohen-Macaulay of length $n$ then so is $P \ast T_{t,n}$.  Equation (\ref{gtreeintro2}) below is proved in \cite{sw2}, and equation (\ref{gtreeintro1}) follows quickly from (\ref{gtreeintro2}) and \cite[Corollary 2.4]{sw2}.  We have
\begin{equation}\label{gtreeintro1} \rk \tilde H_{n-1}(B_n(q)^- \ast T_{t,n-1})= t\sum_{\sigma \in \mathcal D_n} q^{\binom n 2 - \maj(\s) + \exc(\s)}t^{\exc(\sigma)}\end{equation} and
\begin{equation} \label{gtreeintro2} \rk \tilde
H_{n-1}((B_n(q) \ast T_{t,n})^-)= t \sum_{\sigma \in \sg_n} q^{\binom n 2 - \maj(\s) + \exc(\s)} t^{\exc(\sigma)}.\end{equation}

One can also find in \cite{sw2} type BC analogues of the results mentioned above, where $B_n$ and $B_n(q)$ are replaced, respectively, by the poset of faces of the $n$-crosspolytope and the poset of totally isotropic subspaces of a $2n$-dimensional vector space over $\ff_q$ equipped with a nondegenerate alternating bilinear form, and $\dd_n$ is replaced by the set of elements of the Weyl group of type BC that act as derangements on the set of vertices of the crossplytope.  In \cite{mr}, P. Muldoon and M. Readdy prove  an analog of  (\ref{jon}) that involves the poset of faces of the $n$-cube.  

As was mentioned above, the results of Bj\"orner and Welker \cite{bw} are concerned with  Cohen-Macaulayness of Rees products.  It turns out that analogous results for lexicographic shellability can be obtained and  utilized to obtain enumerative results.      Definitions and basic facts about lexicographic shellability are given in Section \ref{lexshellsec}.

Let $P$ be a pure poset of length $n$. For $S \subseteq [0,n]:= \{0,1,\dots,n\}$, the {\em rank-selected subposet} $P_S$ is the subposet of $P$ consisting of all $x \in X$ satisfying $r_P(x) \in S$.  If $P$ is  lexicographically shellable then $P$ is Cohen-Macaulay, as is every rank-selected subposet of $P$ (cf. \cite{bj}).  Thus, for all $S \subseteq [0,n]$, the homology of $P_S$ is determined by the Betti number
\[
\beta(P_S):=\rk \tilde H_{|S|-1}( P_S).
\]

Let $\hz_T$ be the minimum element of $T_{t,n}$.  Note that  if $P$ has a unique minimum element $\hat 0_P$ then  the poset $P \ast T_{t,n}$ has a minimum element $(\hz_P,\hz_T)$, but no maximum element.  Write $(P \ast T_{t,n})^+$ for the poset $P \ast T_{t,n}$ with a maximum element appended.  In Section \ref{lexshellsec} we show that if $P$ is lexicographically shellable then so is $(P \ast T_{t,n})^+$ for all $t$.  (In fact, we prove a stronger result, see Theorem \ref{ELth}.)

We call $S \subseteq \nn$ {\em stable} if there is no $i \in \nn$ such that $\{i,i+1\}  \subseteq S$.  For $X \subseteq \nn$, we write $\psx(X)$ for the set of all stable $S \subseteq X$. We use the lexicographic shellings described in Section \ref{lexshellsec} to prove in Section \ref{ascfreesec} that, for pure,  lexicographically shellable $P$ of length $n$, 
\begin{equation} \label{introbettieq}
\beta((P \ast T_{t,n})^-)=\sum_{S \in \psx([n-2])}\beta(P_{[n-1] \setminus S})t^{|S|+1}(1+t)^{n-2|S|-1}.
\end{equation}
In fact, we prove in Section \ref{ascfreesec} several formulae similar to (\ref{introbettieq}) involving either $P^-\ast T_{t,n}$ or $(P \ast T_{t,n})^-$.   (More general versions of these formulae in which the only requirement on P is that it be pure will appear in a forthcoming paper.)

In Sections \ref{chainsec}, \ref{qbnsec}, \ref{psec} and \ref{noncrosssec} we apply our results from Sections \ref{lexshellsec} and \ref{ascfreesec} to obtain enumerative results.

The Boolean algebra $B_n$ is the direct product of $n$ copies of the chain $C_1$.  In Section \ref{chainsec} we prove  generalizations of  the $q=1$ cases of (\ref{gtreeintro1}) and (\ref{gtreeintro2}) in which we replace $B_n$ with an arbitrary product of finite chains. 
Let  $\mu=(\mu_1,\ldots,\mu_k)$ be a weak composition of $n$ into $k$ parts, that is a $k$-tuple of nonnegative integers whose sum is $n$.  The product poset $B_\mu:=\prod_{i=1}^{k}C_{\mu_i}$ is  pure of length $n$.  It is well known that $B_\mu$ is lexicographically shellable.

Let $M(\mu)$ be the multiset in which each $i \in [k]$ appears with multiplicity $\mu_i$.  A {\em multiset permutation} of $M(\mu)$ is a $2 \times n$ array $(a_{ij})$ such that 
\begin{itemize}
\item the multisets $\{a_{1j}:j \in [n]\}$ and $\{a_{2j}:j \in [n]\}$ are both equal to $M(\mu)$,
\item $a_{1,j} \le a_{1,j+1}$ for all $j \in [n-1]$.
\end{itemize}
Let $w=(a_{ij})$ be a multiset permutation of $M(\mu)$.  We say $w$ is  {\em multiset derangement} of $M(\mu)$ if
\begin{itemize}
\item $a_{1j} \neq a_{2j}$, for all $j \in [n]$.
\end{itemize}
We say $w$ is a {\em Smirnov word} on $M(\mu)$ if 
\begin{itemize}
\item $a_{2j} \neq a_{2,j+1}$ for all $j \in [n-1]$.
\end{itemize}
An {\em excedance} of $w$ is any $j \in [n-1]$ such that $a_{2j}>a_{1j}$.  A {\em descent} of $w$ is any $j \in [n-1]$ such that $a_{2j}>a_{2,j+1}$.  We write $\Exc(w)$ for the set of excedances of $w$, $\Des(w)$ for the set of descents of $w$, and $\des(w)$ and $\exc(w)$, respectively,  for $|\Des(w)|$ and $|\Exc(w)|$.   

Our main results in Section \ref{chainsec} say that if ${\mathcal M}{\mathcal D}_{M(\mu)}$ and ${\mathcal S}{\mathcal W}_{M(\mu)}$ are, respectively, the sets of  multiset derangements and Smirnov words on $M(\mu)$ then, for all $t \in \pp$,
\begin{equation} \label{mddth}
\beta(B_\mu^- \ast T_{t,n-1})=\sum_{w \in {\mathcal M}{\mathcal D}_{M(\mu)}}t^{1+\exc(w)},
\end{equation}
and
\begin{equation} \label{swth}
\beta((B_\mu \ast T_{t,n})^-)=\sum_{w \in {\mathcal S}{\mathcal W}_{M(\mu)}}t^{1+\des(w)}.
\end{equation}
When  $M(\mu) $ is the set $[n]$, equation (\ref{mddth}) is the $q=1$ case of  (\ref{gtreeintro1}).  Since $\des$ and $\exc$ are equidistributed on the symmetric group $\sg_n$,  equation (\ref{swth}) is the $q=1$ case of (\ref{gtreeintro2}).

In Section \ref{qbnsec} we revisit the Rees products $B_n(q) \ast T_{t,n}$ that were studied in \cite{sw2}.  Comparing (\ref{gtreeintro2}) with a formula for $\rk \tilde{H}_{n-1}((B_n(q) \ast T_{t,n})^-)$ obtained using the techniques developed herein, we exhibit a permutation statistic called ${\aid}$ such that the pair $(\aid,\des)$ is equidistributed on $\mathfrak S_n$ with the pair $(\maj,\exc)$.

In Section \ref{psec} we aim for $p$-analogues of the results in Section \ref{chainsec}. Given a weak composition  $\mu=(\mu_1,\ldots,\mu_k)$ of $n$, a natural choice for a $p$-analogue to the poset $B_\mu$ is the lattice $B_\mu(p)$ of subgroups of the abelian $p$-group $\bigoplus_{j=1}^{k}\zz/p^{\mu_j}\zz$.  We examine $B_\mu(p)^-\ast T_{t,n-1}$ and $(B_\mu(p) \ast T_{t,n})^-$.  Here our results are less than optimal.  We show that there exist statistics $s_1,s_2:\pp^n \rightarrow \nn$ such that 
\begin{equation} \label{bmup1}
\beta((B_\mu(p)^-\ast T_{t,n-1})=\sum_{w \in {\mathcal M}{\mathcal D}_{M(\mu)}}p^{s_1(w)}t^{1+\exc(w)},
\end{equation}
and
\begin{equation} \label{bmup2}
\beta((B_\mu(p) \ast T_{t,n})^-)=\sum_{w \in {\mathcal S}{\mathcal W}_{M(\mu)}}p^{s_2(w)}t^{1+\des(w)}.
\end{equation}
However, we lack natural combinatorial interpretations for $s_1$ and $s_2$.

In Section~\ref{noncrosssec} we consider the lattice $\NC_n$ of nonncrossing partitions of $[n]$, which is known to be  lexicographically shellable.  We show that
\begin{equation} \label{nc1}
\beta((\NC_{n+1} \ast T_{t,n})^-)=\frac{1}{n+1}\sum_{k=0}^{n-1}{{n-1} \choose {k}}\sum_{w \in [n+1]^{n-k}}t^{\des(w)+k},
\end{equation}
and
\begin{equation} \label{nc2}\beta(\NC_{n+1} ^- * T_{t,n-1}) =\end{equation} $$  (-1)^{n}+ \frac {1}{n+1}\sum_{r=0}^{n-1}(-1)^r \binom {n+1} r \sum_{k=0}^{n-1-r} {n-1-r \choose k} \sum_{w \in [n+1]^{n-k-r}} t^{\des(w)+k} . $$
Equation (\ref{nc1}) reduces to a particularly nice enumerative formula when $t$ is set equal to $1$, namely
$$\beta((\NC_{n+1} \ast C_n)^-) = (n+2)^{n-1}.$$

Proofs of the various identities stated above involve symmetric function formulae for generating functions for  words with no double descent, words with no double ascent, Smirnov words, and multiset derangements, keeping track of  descents, ascents, descents and excedances, respectively.  The formula involving  Smirnov words follows from work in \cite{sw2}, while the remaining formulae are due to Ira Gessel.  We give all of these formulae in Section \ref{symmsec}.

\part{Lexicographical Shellability}

 \section{Edge labelings of Rees products} \label{lexshellsec}
 
After reviewing some basic facts from the theory of lexicographic
shellability  (cf. \cite{bj, bjwa1, bjwa2, bjwa3, w1}), we will present our main results on lexicographic shellability of Rees products.  Let $P$ be a  {\em bounded} poset, i.e., a poset with a unique minimum element and a unique maximum element, and  let $\Cov(P)$
be the set of pairs $(x,y) \in P \times P$ such that $y$ covers
$x$ in $P$.  Let $L$ be another poset and let $W$ be the set of
all finite sequences of elements of $L$.  The given partial ordering of
$L$ induces a lexicographic ordering $\preceq$ on $W$, which is also a partial order.
 An {\it edge
labeling} of $P$ by $L$ is a function $\lambda:\Cov(P) \rightarrow
L$.  Given such a function $\lambda$ and a saturated chain
$C=\{x_1<\ldots<x_m\}$ from $P$, we write $\lambda(C)$ for
$(\lambda(x_1,x_2),\ldots,\lambda(x_{m-1},x_m)) \in W$.  An {\it
ascent} in $C$ is any $i \in [m-1]$ satisfying
$\lambda(x_i,x_{i+1})\le\lambda(x_{i+1},x_{i+2})$. We say
$\lambda$ is {\it weakly increasing} on $C$ if each  $i \in [m-1]$ is an
ascent in $C$.  The edge labeling $\lambda$ is an {\it
EL-labeling} of $P$ if whenever $x<y$ in $P$ there is a unique
maximal chain $C$ in the interval $[x,y]$ on which $\lambda$ is
weakly increasing and for all other maximal chains $D$ in $[x,y]$ we have
$\lambda(C) \prec \lambda(D)$.  A bounded poset that admits an EL-labeling is said to be
EL-shellable.

The notion of EL-shellability for {\em pure}  posets was introduced by Bj\"orner in \cite{bj}.  A more general concept called CL-shellability,  introduced  by
Bj\"orner and Wachs  in \cite{bjwa1}, also associates label sequences with maximal chains of a poset. We will not define CL-labelings here.  Both notions were subsequently extended to all  bounded posets by Bj\"orner and Wachs in \cite{bjwa3}.     All of our results in this section and the next section hold  for CL-labelings as well as EL-labelings.  For the sake of simplicity we state and prove them only for EL-labelings.  The proofs for CL-labelings are virtually the same as those for EL-labelings.

Given an EL-labeling $\lambda$  on
$P$, we call a maximal chain $C$ from $P$ {\it ascent free} if its label sequence
contains no ascent.  The {\em descent set} of a maximal chain $x_0 < x_1 < \dots < x_n$ is defined to be the set $\{i \in [n-1] : \lambda(x_{i-1},x_i) \not\le \lambda(x_{i},x_{i+1}) \}$.    Thus a maximal chain is ascent free if and only if its  descent set is $[n-1]$. 

One of the main results in the theory of
lexicographic shellability  is the following result.

\begin{thm}[Bj\"orner and Wachs \cite{bjwa3}]
Let $\lambda$ be an EL-labeling  of a bounded   poset $P$ with minimum $\hat 0$ and maximum $\hat 1$.  
Then $P\setminus\{\hat 0,\hat 1\}$ is homotopy equivalent to a wedge of  spheres, where for each $k \in \N$ the number of spheres  of dimension $(k-2)$ is the number of ascent free maximal chains of length $k$.
\label{elth}
\end{thm}

We will also need the following basic result.   Given a pure poset $P$ of length $n$ and a set $S \subseteq [0,n]$, recall that 
the rank selected subposet is defined by $$P_S := \{x \in P : r_P(x) \in S\}.$$

\begin{thm}[Bj\"orner \cite{bj}] \label{rankselth} Let $\lambda$ be an EL-labeling  of a bounded    pure poset $P$ of length $n$.    For $S \subseteq [n-1]$, let $c(S)$ be the number of
maximal chains in $P$ having descent set $S$ with respect to $\lambda$. 
Then $P_S $ has the homotopy type of a wedge of $c(S)$ spheres of
dimension $|S|-1$.
\end{thm}

 Given a poset $P$, by $\hat P$ we mean the poset $P$ with a new minimum element $\hat 0$ and a new maximum element $\hat 1$ attached even if $P$  already has such elements.   Given a poset $P$ with a minimum element $\hat 0$, we say that an edge labeling $\lambda:\Cov(P) \rightarrow L$ is a semi-EL-labeling if $[\hat 0, m]$  is an EL-labeling for each maximal element $m$ of $P$.  Note if $P$ is bounded then $\lambda$ is a semi-EL-labeling if and only if it is an EL-labeling.  
 Recall that we defined $T_{t,n}$ to be  the poset whose Hasse diagram is the complete $t$-ary tree of height $n$ with the root at the bottom.  The edge labeling in which all the edges in $\Cov(T_{t,n}) $ are  labeled with $1$ is clearly a semi-EL-labeling of $P$.

 \begin{thm} \label{ELth}
Let $P_1$ and $P_2$ be   semipure posets of the same length.  Assume also that $P_2$ has a minimum element  $\hat 0_2$.  
Let $\lambda_1:\Cov(\hat P_1) \rightarrow L_1$ be an EL-labeling  of $\hat P_1$ and let $\lambda_2:\Cov(P_2) \rightarrow L_2$ be a semi-EL-labeling of $P_2$.   Let $\hat 0_1$ denote the minimum element of $\hat P_1$ and let $\hat 1_1$ denote the maximum element.  Let   $(\hat 0_1, \hat 0_2)$ denote the minimum element of 
$\wh{P_1*P_2}$ and let $\hat 1$ denote the maximum element.  Define the edge labeling $$\lambda:\Cov(\wh{P_1*P_2})
\rightarrow  L_1  \times ( L_2 \uplus \{\hat 0_{L_2}\}) $$ by
\[
\lambda((x,k),(y,l))=\begin{cases} (\lambda_1(x,y),\lambda_2(k,l)) &\mbox{if } k<_{P_2} l\\ 
(\lambda_1(x,y),\hat 0_{L_2}) &\mbox{if }k=l\ \end{cases}
\]
for $(y,l) < \hat 1$, and
\[ \lambda((x,k), \hat 1) = \ (\lambda_1(x,\hat 1_1), \hat 0_{L_2}).\]
Then $\lambda$ is an EL-labeling of $\wh{P_1*P_2}$.
\label{elrees}
\end{thm}

\begin{proof}  {\bf Case 1: } $(x,k) < (y,l)<\hat 1$ in $\wh{P_1 * P_2}$.  Then  $x <y$ in $\hat P_1$ and $k \le l$ in $P_2$.  It follows that there is a unique maximal chain 
$\{x=u_0<\dots<u_m=y\}$ in  $[x,y]$ on
which $\lambda_1$ is weakly increasing and  a unique maximal chain $\{k=c_{0} <\dots< c_{r(l)-r(k)}=l\}$ in  $[k,l]$ on which $\lambda_2$ is weakly increasing.   Let 
$$e_i = \begin{cases} k &\mbox{ for } 0 \le i \le m+r(k) -r(l)\\ c_{i-m-r(k) +r(l)} &\mbox{for  }  m+r(k) -r(l) < i \le m .\end{cases}$$
The labeling $\lambda$ is weakly increasing on the maximal chain $$C:=\{(u_0,e_0)  <  (u_1,e_1)<\cdots < (u_{m},e_{m})\},$$ of the
interval $I:=[(x,k),(y,l)]$.

To establish uniqueness of the maximal chain with weakly increasing labels, suppose that  $\lambda$ is weakly increasing on
the maximal chain $D=\{(v_0,f_0)<\dots<(v_m,f_m)\}$ in the
interval $I$.  Then $\lambda_1$ is weakly increasing on the  chain $\{v_0 < \dots < v_m\}$, which implies that $v_i=u_i$ for all $0 \leq i \leq m$.
Moreover, if $\lambda(D)=((a_1,d_1),\dots,(a_m,d_m))$  then we must have $d_i=\hat 0_{L_2}$ for $1 \leq
i \leq m+r(k)-r(l)$ and $d_i \le d_{i+1} $  in $L_2$ for $m+r(k)-r(l)<i \leq m-1$ .  If $d_i = \hat 0_{L_2}$ then $f_{i-1} = f_i$, and  if  $d_i \in L_2$ then $f_{i-1} $ is covered by $ f_i$ in $P_2$ and $d_i
=\lambda_2(f_{i-1},f_i)$.  It follows that    if $j= m+r(k)-r(l)$ then 
$$k=f_0 = f_1 =\dots = f_j$$ and $$\{ f_j < f_{j+1} < \dots < f_m\}$$ is the unique maximal chain of the interval $[k,l] $ in $ P_2$ for which $\lambda_2$ is weakly increasing.  Thererfore   $f_i=e_i$ for all $i$.

Next we show  that the  maximal chain $C$ of $I$ has a label sequence that    lexicographically precedes the label sequences of all maximal chains of $I$.
Let $D=\{(v_0,f_0)<\dots<(v_m,f_m)\}$  be another maximal chain in
$I$.  Assume that $(u_i,e_i)=(v_i,f_i)$ for  $0 \leq i <t$ but
$(u_t,e_t) \neq (v_t,f_t)$.  We need to show that 
\begin{equation}\label{lexfirsteq}  \lambda((u_{t-1},e_{t-1}), (u_t,e_t) ) < \lambda((v_{t-1},f_{t-1}), (v_t,f_t) )\end{equation} in $ L_1 \times ( L_2 \uplus \{\hat 0_{L_2}\})$

First we handle the case in which
$1\le t \le m+r(k)-r(l)$.  In this case we have   $e_{t-1} =e_t = k$,  which implies
\begin{equation} \label{lameq} \lambda((u_{t-1},e_{t-1}), (u_t,e_t) ) = (\lambda_1(u_{t-1},u_t),\hat 0_{L_2}).\end{equation} If $u_t = v_t$ then $f_t \ne e_t$, which implies that  $f_t$ covers $f_{t-1}= k$ in $P_2$.  Since   $\lambda_2(f_{t-1},f_t) > \hat 0_{L_2} $ and $$\lambda((v_{t-1},f_{t-1}), (v_t,f_t) ) = (\lambda_1(v_{t-1},v_t),\lambda_2(f_{t-1},f_t) ),$$
 (\ref{lexfirsteq}) holds. Now assume  $u_t \ne v_t$.  We have \begin{equation} \label{lam1eq} \lambda_1(u_{t-1},u_t) <
\lambda_1(v_{t-1},v_t)\end{equation}
 in $L_1$.
Indeed, it is a basic property of EL-labelings that if $P$ is a poset with EL-labeling $\lambda$ then for each interval $[x,y]$, if $a$ covers $x$ in the unique  maximal chain of  $[x,y]$ with weakly  increasing labels and $b$ is an atom of $[x,y]$ other than $a$, then $\lambda(x,a) < \lambda(x,b)$ (cf. \cite[Proposition 2.5]{bj}, \cite[Lemma~5.3]{bjwa3}).  Since   $\lambda((v_{t-1},f_{t-1}), (v_t,f_t) ) = (\lambda_1(v_{t-1},v_t),d)$, for some $d \in  L_2 \uplus \hat 0_{L_2}$, the desired inequality  (\ref{lexfirsteq})  follows from (\ref{lameq}) and
 (\ref{lam1eq}).  
 
 Now assume $ m+r(k)-r(l)< t \le m$.  In this case we have 
 $u_t$ and $v_t$ cover $u_{t-1} = v_{t-1}$ in $\hat P_1$, and 
 $e_t$ and $f_t$ cover $e_{t-1}=f_{t-1}$ in $P_2$. It  now follows from the basic property of EL-labelings mentioned in the previous paragraph that either (\ref{lam1eq})   and $\lambda_2(e_{t-1},e_t) \le\lambda_2(f_{t-1},f_t) $ or  $ \lambda_1(u_{t-1},u_t) \le\lambda_1(v_{t-1},v_t)$ and
$\lambda_2(e_{t-1},e_t) < \lambda_2(f_{t-1},f_t) $ hold, which yields the desired conclusion (\ref{lexfirsteq}).

{\bf Case 2:} $(x,k)<\hat 1$ in $\wh{P_1 * P_2}$.  Then  $x \le \hat 1_1$ in $\hat P_1$ and there is a unique maximal chain 
$\{x=u_0<\dots<u_m< \hat 1_1\}$ in  $[x,\hat 1_1]$ on
which $\lambda_1$ is weakly increasing.      The labeling  $\lambda $ is weakly increasing on the maximal chain
$$C:=\{(u_0,k)  <\dots < (u_{m},k) < \hat 1\}$$ of the
interval $[(x,k),\hat 1]$.  To establish uniqueness of the maximal chain with weakly increasing labels, note that  the top label of every maximal chain of $[(x,k),\hat 1]$ is of the form $(\lambda_1(v,\hat 1_1), \hat 0_{L_2})$, where $v$ is a maximal element of  $P_1$ .   Hence  if   $D=\{(v_0,f_0)<\dots<(v_{m^\prime},f_{m^\prime})<\hat 1\}$ is a maximal chain of $[(x,k),\hat 1]$ with weakly increasing labels then  $$\lambda(D) = \{(a_1,\hat 0_{L_2}), \dots , (a_{m^\prime+1}, \hat 0_{L_2}))\}, $$ where $a_1 \le  \dots \le a_{m^\prime+1}$ in $L_1$.
 It follows that $f_i =k$ for all $i=1, \dots, m^\prime$ and $\{v_0 < \dots < v_{m^\prime} <\hat 1_1\}$ is the unique maximal chain of $[x, \hat 1_1]$ with weakly increasing labels.   Hence $m=m^\prime$ and $v_i = u_i$ for all $i=1, \dots, m$.
 
  Now
let $D=\{(v_0,f_0)<\dots<(v_{m^\prime},f_{m^\prime})<\hat 1\}$  be a maximal chain in
$[(x,k),\hat 1]$ that is different from $C$.  We show that the label sequence of $C$ is lexicographically less than that of $D$.  Assume that $(u_i,k)=(v_i,f_i)$ for  $1 \leq i <t$ but
$(u_t,k) \neq (v_t,f_t)$.  We need to show that 
\begin{equation}\label{lexfirsteq2}  \lambda((u_{t-1},k), (u_t,k) ) < \lambda((v_{t-1},f_{t-1}), (v_t,f_t) )\end{equation} in $ L_1 \times ( L_2 \uplus \{\hat 0_{L_2}\})$.
If $u_t = v_t $ then $f_t $ covers $f_{t-1}=k$ in $P_2$.  
We have \begin{eqnarray*} \lambda((u_{t-1},k), (u_t,k) ) &=&  (\lambda_1(u_{t-1},u_t),\hat 0_{L_2}) \\
&= & (\lambda_1(v_{t-1},v_t),\hat 0_{L_2}) \\ &< &(\lambda_1(v_{t-1},v_t),\lambda_2(f_{t-1},f_t)) \\ &=&\lambda((v_{t-1},f_{t-1}), (v_t,f_t) ).\end{eqnarray*} If $u_t \ne v_t $ then by the basic property of EL-labelings mentioned above,
$\lambda_1(u_{t-1},u_t) < \lambda_1(v_{t-1},v_t) $.   It follows that 
\begin{eqnarray*} \lambda((u_{t-1},k), (u_t,k) ) &=&  (\lambda_1(u_{t-1},u_t),\hat 0_{L_2}) \\
&< &(\lambda_1(v_{t-1},v_t),\lambda_2(f_{t-1},f_t)) \\ &=&\lambda((v_{t-1},f_{t-1}), (v_t,f_t) ).
\end{eqnarray*}

\end{proof}

\section{Ascent free chains of $P*T_{t,n}$} \label{ascfreesec} Let $P$ be a  semipure poset of length $n$.
Let $\lambda_P:\Cov( \hat P) \to L_P$ be an EL-labeling of $\hat P$ and let $\lambda_T$ be 
the semi-EL-labeling of $T_{t,n}$ in which each edge has label 1.   In this section we count the ascent free maximal 
chains of $\widehat{P*T_{t,n}}$ under the EL-labeling $\lambda:\Cov(\widehat{P*T_{t,n}})\to (L_P \times \{0<1\})$  described in 
Theorem \ref{ELth}.   

For $j = 0,\dots, m$, let $S_{m,j}$ be the set of sequences $(d_1,\dots, d_m) \in \{0,1\}^m$ such that  $\sum_{i=1}^{m}d_i=j$. Given any maximal chain $D = \{(\hat 0_P, \hat 0_T) < (x_0,f_0) < \dots <(x_m,f_m) <\hat 1\}$ of $\widehat{P*T_{t,n}}$, we have that $\{ x_0 <x_1 <\dots <x_m \}$ is a maximal chain of $P$ and $(r(f_1)-r(f_0), r(f_2)-r(f_1), \dots,r(f_m)-r(f_{m-1})) \in  S_{m,j}$, for some $j$.
Conversely,
given any maximal chain $C =\{x_0 <x_1 <\dots <x_m\}$ of $P$ and any $d\in  S_{m,j}$,
 there is a maximal chain $D = \{(\hat 0_P, \hat 0_T) < (x_0,f_0) < \dots <(x_m,f_m) <\hat 1\}$ of $\widehat{P*T_{t,n}}$ such that 
 $r(f_i)-r(f_{i-1}) = d_i$ for all $i\in [m]$.  Let $[C,d]$ be the set of  all such maximum chains of 
 $\widehat{P*T_{t,n}}$.  
   
  The following propositions clearly hold.
\begin{prop} \label{propnochains} The sets $[C,d]$, where $C $ is a maximal chain of $P$ of length $m$ and $d \in S_{m,j}$ for $j = 0,\dots, m$,  partition the set of maximum chains of 
$\widehat{P*T_{t,n}}$.  Moreover if $d \in S_{m,j}$ then $|[C,d]| = t^j$.
\end{prop}

\begin{prop} \label{propascfree}   Let $$C:=\{x_0  <\dots <x_{m} \} $$ be a maximal chain of $P$ and let $d:= (d_1,\dots,d_m) \in \{0<1\}^m$. Then for each maximal chain $D \in [(C,d)]$ we have $$\lambda(D) = ((\lambda_P(\hat 0_P,x_0), 0), (\lambda_P(x_{0},x_1),d_1),\dots,(\lambda_P(x_{m-1},x_m),d_m), (\lambda_P(x_m,\hat 1_P),0)).$$   Consequently, $D$ is ascent free if and only if $\lambda_P(\hat 0_P,x_{0} ) \not\le \lambda_P(x_{0},x_1) $ and 
\begin{equation} \label{ascfreeeq} \forall i \in [m], \, \,\,\, \lambda_P(x_{i-1},x_i) \le \lambda_P(x_{i},x_{i+1} ) \implies d_i =1 \mbox{ and } d_{i+1} = 0 \end{equation}
holds. Here we have set $x_{m+1} := \hat 1_P$ and $d_{m+1} :=0$.
\end{prop}

 Given a word $w=w_1\cdots w_n$ over a partially ordered alphabet $A$, we say $i\in [n-1]$ is an ascent of $w$ if $w_i \le w_{i+1}$ and  that $i \in [n-2]$ is a double ascent if  $w_i \le w_{i+1} \le w_{i+2}$.   Let $\asc(w)$ denote
the number of ascents of $w$ and $$\NDA_n(A):= \{w \in A^{n} : w \mbox{  has no double ascents} \}.$$

We are now ready to count the ascent free maximal chains.  We begin with the case in which the semipure poset $P$ has a unique maximum element.  In this case $P$ must necessarily be a pure poset of length $n$.  All maximal chains of $\hat P$ have length $n+2$ and must have an ascent at $n+1$ under the EL-labeling.   We leave it to the reader to observe that Propositions~\ref{propnochains} and~\ref{propascfree} imply the following result.

\begin{thm} \label{maxth} If $P$ has a unique maximum element then the number of ascent free maximal chains of $\widehat{P*T_{t,n}}$ of length $n+2$ under the EL-labeling of Theorem~\ref{ELth} is given by
$$ \sum_{\scriptsize\begin{array}{c}w \in  \NDA_{n+1}(L_P)\\ w_1 \not\le w_2 \\ w_{n} \not\le w_{n+1}\end{array} }  c(w) t^{\asc(w)+1} (1+t)^{n-1-2\asc(w)},$$ where $ c(w)$ is the number of maximal chains of $P \uplus {\hat 0_P}$  with label sequence $w$.
\end{thm}

In the general case in which it is not assumed  that $P$ has a unique maximum element, we have the following result which also is a consequence of  Propositions~\ref{propnochains} and~\ref{propascfree}.

\begin{thm} Let $m \in \N$.  Then the number  ascent free maximal chains of $\widehat{P*T_{t,n}}$ of length $m+2$ under the EL-labeling of Theorem~\ref{ELth} is given by
$$ \sum_{\scriptsize\begin{array}{c}w \in  \NDA_{m+2}(L_P)\\ w_1 \not\le w_2 \\ w_{m+1} \not\le w_{m+2}\end{array} } c(w) t^{\asc(w)} (1+t)^{m-2\asc(w)}$$
$$ +  \sum_{\scriptsize\begin{array}{c}w \in  \NDA_{m+2}(L_P)\\ w_1 \not\le w_2 \\ w_{m+1}\le w_{m+2}\end{array} } c(w) t^{\asc(w)} (1+t)^{m+1-2\asc(w)},$$
where  $c(w)$ is the number of  maximal chains of $\hat P$ of length $m+2$ 
with label sequence $w$.
\end{thm}

Note that if $P$ has a unique minimum element  then $P*T_{t,n}$ has unique minimum element, which implies that $P*T_{t,n}$ is contractible. Hence the number of ascent free maximal chains of $\widehat{P*T_{t,n}}$ has to be 0.  This is corroborated  by $c(w) = 0$ if $w_1 \not\le  w_2$, which follows from the fact that there is only one maximal chain in each interval $[\hat 0,a]$ of $\hat P$, where $a$ is an atom of $P$.  Therefore in the case that $P$ has a unique minimum element,  it is more interesting to consider the number of ascent free chains of the interval $ (P*T_{t,n})^+$ of $\widehat {P*T_{t,n}}$.  The following results also follow from Propositions~\ref{propnochains} and~\ref{propascfree}.

\begin{thm} \label{minmaxth} If $P$ has both a unique minimum element and a unique maximum element then the number of ascent free maximal chains of $(P*T_{t,n})^+$  under the EL-labeling of Theorem~\ref{ELth} is given by
$$ \sum_{\scriptsize\begin{array}{c}w \in  \NDA_{n}(L_P)\\ w_{n-1}\not\le w_{n}\end{array} }  c(w) t^{\asc(w)+1} (1+t)^{n-1-2\asc(w)},$$ where $ c(w)$ is the number of maximal chains of $P $  with label sequence $w$.
\end{thm}

\begin{thm} Let $m \in \N$.   If $P$ has a unique minimum element then the number of ascent free maximal chains of $(P*T_{t,n})^+$ of length $m+1$ under the EL-labeling of Theorem~\ref{ELth} is given by
$$\sum_{\scriptsize\begin{array}{c}w \in  \NDA_{m+1}(L_P)\\ w_{m} \not\le w_{m+1}\end{array} } c(w) t^{\asc(w)} (1+t)^{m-2\asc(w)} $$  $$+ \sum_{\scriptsize\begin{array}{c}w \in  \NDA_{m+1}(L_P)\\ w_{m}\le w_{m+1}\end{array} } c(w) t^{\asc(w)} (1+t)^{m+1-2\asc(w)},$$ where $ c(w)$ is the number of maximal chains of $P^+ $ of length $m+1$ with label sequence $w$.
\end{thm}

For pure $P$ we can restate the above results by applying Theorem~\ref{rankselth}.  
We need to recall the following terminology and notation. A set of integers is {\it stable} if it contains no two consecutive integers.  For $X\subseteq \Z$,  the set of all stable subsets of $X$ is denoted by  $\psx(X)$. For $i \le j \in \N$, let
$[i,j] :=\{i,i+1,\dots,j\}$ and $[j]:=[1,j]$.  If $P$ is a poset of length $n$ let $$\beta (P): = \rk \tilde H_{n}(P).$$ If $P$ has a unique minimum element $\hat 0$ let
$$P^- := P \setminus \{\hat 0 \}.$$

\begin{cor} \label{mobcor1} Let $P$ be a pure  poset of length $n$ such that  $\hat P$ is  EL-shellable.
Assume that $P$ has a unique maximum element.  Then
\begin{equation} \label{minmaxeq1} \beta(P * T_{t,n}) = \sum_{S \in \psx([n-2])} \beta(P_{[0,n-1]\setminus S} ) t^{|S|+1} (t+1)^{n-2|S|-1 }.\end{equation}
If $P$  also has a unique minimum element then
\begin{equation}\label{minmaxeq2} \beta((P * T_{t,n})^-) = \sum_{S \in \psx([n-2])} \beta(P_{[n-1]\setminus S}) t^{|S|+1} (t+1)^{n-2|S|-1 }.\end{equation}
\end{cor}

\begin{cor} \label{mobcor2}  Let $P$ be a pure  poset of length $n$ such that  $\hat P$ is  EL-shellable.   Then  $$\beta(P * T_{t,n})=$$
$$\sum_{S \in \psx([n-1])} \beta ( P_{[0,n]\setminus S}) t^{|S|} (t+1)^{n-2|S| } +  \sum_{S \in \psx([n-2])}       \beta ( P_{[0,n-1]\setminus S}) t^{|S|+1} (t+1)^{n-2|S|-1 } .$$
If $P$  has a unique minimum element then
$$\beta((P * T_{t,n})^-) =$$ $$   \sum_{S \in \psx([n-1])} \beta ( P_{[n]\setminus S}) t^{|S|} (t+1)^{n-2|S| } +  \sum_{S \in \psx([n-2])}       \beta( P_{[n-1] \setminus S} ) t^{|S|+1} (t+1)^{n-2|S|-1 } .$$
\end{cor}

\part{Applications}

\section{Symmetric function preliminaries} \label{symmsec}

Let $h_n= h_n(x_1,x_2,\dots)$ denote the complete homogenous symmetric function of degree $n$  in indeterminants ${\bf x}:=x_1,x_2\dots$ and  $e_n=e_n(x_1,x_2,\dots)$  denote the elementary symmetric function  of degree $n$ in indeterminants ${\bf x}$. That is 
$$h_n({\bf x}) := \sum_{1 \le i_i  \le \dots \le i_n} x_{i_1} \cdots x_{i_n}  \mbox { and } e_n({\bf x}) := \sum_{1 \le i_i  < \dots < i_n} x_{i_1} \cdots x_{i_n}. $$ Also let 
$$[n]_t := 1+t+ \dots + t^{n-1}.$$  In this section we will discuss various combinatorial  interpretations of  variations of the symmetric function 
$${{ \sum_{i\ge 0} h_i z^i} \over {1 - \sum_{i \ge 2} t[i-1]_t h_i z^i}} \, , $$
which play a key role in the proofs of the results in the subsequent sections.  These and other interpretations are discussed in \cite[Section 7]{ShWa}.

 Let $w=w_1\cdots w_n \in \pp^n$.
Recall that we say $i\in [n-1]$ is an ascent of $w$ if $w_i \le w_{i+1}$ and  that $i \in [n-2]$ is a double ascent if  $w_i \le w_{i+1} \le w_{i+2}$.   Recall that $\asc(w)$ denotes
the number of ascents of $w$ and $$\NDA_n:= \NDA_n(\pp)= \{w \in \pp^{n} : w \mbox{  has no double ascents} \}.$$  Similarly,   $i\in [n-1]$ is a descent of $w$ if $w_i> w_{i+1}$ and  $i \in [n-2]$ is a double descent if  $w_i > w_{i+1} > w_{i+2}$.   Let $\des(w)$ denote
the number of descents of $w$ and $$\NDD_n:= \NDD_n(\pp)= \{w \in \pp^{n} : w \mbox{  has no double descents} \}.$$
We write ${\bf x}_w$ for $x_{w_1} \cdots x_{w_n}$.  

We begin by presenting the following interpretations due to Gessel, see Theorem~7.3 of \cite{ShWa}.  (Gessel's original proofs  will appear in \cite{GW}.)

  \begin{equation}\label{ges1} 1+\sum_{n \ge 1} z^n \sum_{\scriptsize\begin{array}{c} w \in \NDA_n\\ w_1 >w_2\\ w_{n-1} >w_n \end{array}}  t^{\asc(w)} (1+t) ^{n-2-2\asc(w)} {\bf x}_w= {{1} \over {1 - \sum_{i \ge 2} t[i-1]_t e_i z^i}} \,\, ,\end{equation} 
 \begin{equation} \label{ges2} 1+\sum_{n \ge 1} z^n \sum_{\scriptsize\begin{array}{c} w \in \NDA_n\\ w_{n-1} >w_n \end{array}}  t^{\asc(w)} (1+t) ^{n-1-2\asc(w)} {\bf x}_w= {{ \sum_{i\ge 0} e_i z^i} \over {1 - \sum_{i \ge 2} t[i-1]_t e_i z^i}} \,\, ,\end{equation}

  \begin{equation}\label{ges3} 1+\sum_{n \ge 1} z^n \sum_{\scriptsize\begin{array}{c} w \in \NDD_n\\ w_1 \le w_2\\ w_{n-1} \le w_n \end{array}}  t^{\des(w)} (1+t) ^{n-2-2\des(w)} {\bf x}_w= {{1} \over {1 - \sum_{i \ge 2} t[i-1]_t h_i z^i}} \,\, ,\end{equation} 
 \begin{equation} \label{ges4} 1+\sum_{n \ge 1} z^n \sum_{\scriptsize\begin{array}{c} w \in \NDD_n\\ w_{n-1} \le w_n \end{array}}  t^{\des(w)} (1+t) ^{n-1-2\des(w)} {\bf x}_w= {{ \sum_{i\ge 0} h_i z^i} \over {1 - \sum_{i \ge 2} t[i-1]_t h_i z^i}} \,\, .\end{equation}  
 
Next we present an interpretation due to Shareshian and Wachs \cite{ShWa}.  A barred word of length $n$ over  alphabet $A$ is an element of $(A \times \{0,1\})^n$.  We visualize barred words as words over $A$ in which some of the letters are barred; $(a,1)$ is a barred letter and $(a,0)$ is an unbarred letter.   If $w$ is a barred word then $|w|$ denotes the word $w$ with the bars removed.  Similarly,  let $ |a|=|\bar{a}|=a$.  If $\alpha$ is a barred or unbarred letter, we refer to $|\alpha|$ as the {\it absolute value} of  $\alpha$.  For a barred word $w$, let $\bars(w)$ denote the number of barred letters of $w$. Let $ W_n$ be the set of barred words $w=w_1\cdots w_n$   of length $n$ over $\pp$ satisfying 
\begin{enumerate}
\item $w_n$ is unbarred
\item for all $i \in[n-1]$, if $|w_i| < |w_{i+1}|$ then $w_i $ is unbarred 
\item for all $i \in[n-1]$,  if $|w_i| > |w_{i+1}|$ then $w_i$ is barred.
\end{enumerate}
Elements of $W_n$ are called {\em banners} in \cite[Section 3]{ShWa}, where it is shown that
\begin{equation}\label{bannereq} 1+ \sum_{n\ge 1} z^n \sum_{w \in  W_n} t^{\bars(w)} {\bf x}_{|w|} = {{ \sum_{i\ge 0} h_i z^i} \over {1 - \sum_{i \ge 2} t[i-1]_t h_i z^i}} \,\, .\end{equation}

We will also need an  interpretation due to Askey and Ismail \cite{ai}  and one due to Stanley (personal communication, see Theorem~7.2 of \cite{ShWa}).
Given a finite multiset  $M$  over $\pp$, let $\sg_M$ denote the set of multiset permutations of $M$.  Recall that we can 
write  $w \in \sg_M$  in two-line notation as  a $2 \times |M|$  array $(w_{i,j}) $ whose top row is a weakly increasing 
arrangement of the multiset $M$ and whose bottom row is an arbitrary arrangement of $M$.   By supressing 
the top row, we write $w$ in one-line notation as the word, $w_{1} \dots w_{|M|}$, where $w_i := w_{2,i}$.  If $w \in \sg_M$ we say that $w$ has 
length $|M|$.  An excedance of a multiset permutation  $w = (w_{i,j})$, written in two-line notation,  is a column 
$j$ such that $w_{1,j} < w_{2,j}$. Let $\exc(w)$ be the number of excedances of $w$.    

Recall that  $w = (w_{i,j}) \in \sg_M$ is  a multiset derangement if each of the columns of  $w$ have distinct entries, i.e,. $w(1,j)  \ne w(2,j)$ for all $j=1,\dots,|M|$. 
  For example, if
 $$w = \bmatrix 1 & 1& 1& 2& 3& 3 & 4\\ 3 & 2 & 3& 1& 4& 1& 1 \endbmatrix $$
then $w$ is a multiset derangement in $\sg_{\{1^3,2,3^2,4\}}$  and $\exc(w) = 4$.

Now let $\mathcal {MD}_n$ be the set of all multiset derangements of length $n$. 
Askey and Ismail \cite{ai} (see also \cite{kz}) 
proved the following $t$-analog of MacMahon's \cite[Sec. III, Ch. III]{mac1}
result on multiset derangements  \begin{equation}\label{macderang}  \sum_{n \ge 0} z^n  \sum_{w\in \mathcal {MD}_{n}}  t^{\exc(w)}{\mathbf x}_w = 
{ 1\over 1 - \sum_{i \ge 2} t[i-1]_t e_i z^i}.\end{equation}

Recall from Section~\ref{secintro} that a multiset permutation $w = w_1 \cdots w_n \in \sg_M$ is  called a Smirnov word if it has no adjacent repeats, i.e. $w_i \ne w_{i+1}$ for all $i = 1,\dots,n-1$.   
Let $\mathcal {SW}_n$ be the set of all Smirnov words of length $n$.   Stanley (see Theorem~7.2 and (7.7) of \cite{ShWa}) observed that  the following
 $t$-analog of a result of Carlitz, Scoville and Vaughan \cite{csv}
   \begin{equation}\label{smirnov}  
\sum_{n\ge 0} z^n \sum_{w\in \mathcal {SW}_{n}}  t^{\des(w)}\mathbf x_w  = 
{{ \sum_{i\ge 0} e_i z^i} \over {1 - \sum_{i \ge 2} t[i-1]_t e_i z^i}} \end{equation}
is equivalent to  (\ref{bannereq}) by P-partition reciprocity (\cite[Section 4.5]{st4}).

\section{Chain product analog of $B_n$} \label{chainsec}

 In this section we generalize the $q=1$ case of  (\ref{gtreeintro1}) and (\ref{gtreeintro2}) by utilizing the results of the previous section.  Given a weak composition $\mu:=(\mu_1,\dots,\mu_k)$  of $n$, let $B_\mu$ denote the product of chains $C_{\mu_1} \times \dots \times C_{\mu_k}$. Recall  that 
   $M(\mu)$ denotes the mulitset $ \{1^{\mu_1},\dots , k^{\mu_k}\}$.
Given a multiset $M$, let $\mathcal{MD}_{M}$ be the set of multiset derangements of the multiset $M$
 and let $\mathcal{SW}_{M}$ be the set of Smirnov words that are multiset permutations of $M$.

    \begin{thm} \label{multirees} Let $\mu$  be a composition of $n$.  Then $B_\mu^- * T_{t,n-1}$ and $(B_\mu * T_{t,n})^-$ have the homotopy type of a wedge of
$(n-1)$-spheres.  The numbers of spheres in these wedges are,
respectively,
 \begin{equation} \label{deraneq} \beta(B_\mu^- * T_{t,n-1}) = \sum_{w \in \mathcal{MD}_{M(\mu)} } t^{\exc(w)+1}\end{equation}
 and
\begin{equation}\label{smireq} \beta( (B_\mu * T_{t,n})^-)= \sum_{w \in \mathcal{SW}_{M(\mu)} } t^{\des(w)+1}.
\end
{equation}

\end{thm}

\begin{proof} 
 We begin by applying Theorem~\ref{maxth} to $P:= B_\mu^-$, which has length $n-1$.  Let  $k=l(\mu)$.  There is a well-known  EL-labeling of $B_\mu$ in which the edge $$((x_1,\dots,x_i,\dots,x_k), (x_1,\dots,x_{i+1},\dots,x_k))$$ is labeled by $i$.  Here $L_P $ is the totally ordered set $ \{1<2<\dots <k\}$.  Hence $ \widehat{B_\mu^- * T_{t,n-1}} $  has an EL-labeling as described
in Theorem~\ref{ELth}.
 The label sequence of each maximal chain of $P\uplus \{\hat 0 \} = B_\mu$ is a permutation of the multiset $M(\mu)$. Moreover each mulitset permutation occurs exactly once as the label sequence of a maximal chain.  So  $c(w) = 1$ if  $w\in \mathfrak S_{M(\mu)}$ and $c(w) = 0$ if $w \in [k]^n - \mathfrak S_{M(\mu)}$.  
 It  follows from Theorem~\ref{maxth} that  the number of ascent free maximal chains of $ \widehat{B_\mu^- * T_{t,n-1}}$ under the given labeling is 
$$\sum_{\scriptsize\begin{array}{c} w \in \NDA_n \cap \mathfrak S_{M(\mu)} \\ w_1 > w_2 \\ w_{n -1} > w_n \end{array}} t^{\asc(w) +1}(1+t)^{n-2-2\asc(w)}.$$ 
Similarly by Theorems~\ref{ELth} and~\ref{minmaxth} with $P= B_\mu$, the poset $(B_\mu * T_{t,n})^+$ has an EL-labeling for which the number of ascent free maximal chains is 
$$\sum_{\scriptsize\begin{array}{c} w \in  \NDA_n \cap \mathfrak S_{M(\mu)}\\ w_{n -1} > w_n \end{array}} t^{\asc(w) +1}(1+t)^{n-1-2a(w)}.$$  Hence by Theorem~\ref{elth}, the posets $B_\mu^- * T_{t,n-1}$ and $(B_\mu * T_{t,n})^-$ have the homotopy type of a wedge of $(n-1)$-spheres and the top Betti numbers are given by
$$\beta(B_\mu^- * T_{t,n-1}) =\sum_{\scriptsize\begin{array}{c} w \in \NDA_n \cap \mathfrak S_{M(\mu)} \\ w_1 > w_2 \\ w_{n -1} > w_n \end{array}} t^{\asc(w) +1}(1+t)^{n-2-2\asc(w)},$$ and 
$$\beta( (B_\mu * T_{t,n})^-) =\sum_{\scriptsize\begin{array}{c} w \in \NDA_n \cap \mathfrak S_{M(\mu)} \\ w_{n -1} > w_n \end{array}} t^{\asc(w) +1}(1+t)^{n-1-2\asc(w)}.$$  

 By combining  (\ref{macderang}) and (\ref{ges1}) we obtain 
 \begin{equation}\label{gesmac}\sum_{w \in \mathcal {MD}_M} t^{\exc(w) } = \sum_{\scriptsize\begin{array}{c} w \in \NDA_n \cap \sg_M \\ w_1>w_2 \\ w_{n-1} > w_n \end{array}} t^{\asc(w)} (1+t) ^{n-2-2\asc(w)}, \end{equation}
 and by combining (\ref{smirnov}) and (\ref{ges2}) we obtain
\begin{equation} \label{gesstan} \sum_{w \in \mathcal {SW}_M} t^{\des(w) } = \sum_{\scriptsize\begin{array}{c} w \in \NDA_n \cap \sg_M
\\ w_{n-1} > w_n \end{array}} t^{\asc(w)} (1+t) ^{n-1-2\asc(w)}, \end{equation}
for all multisets $M$ on $\pp$ of size $ n$.
Equations (\ref{deraneq}) and (\ref{smireq}) now follow from (\ref{gesmac}) and (\ref{gesstan}), respectively. \end{proof}

\begin{remark} When $M=\{1^n\}$, equation (\ref{gesstan}) reduces to a result of Foata and Sch\"utzenberger \cite{fs}, which is used to show that the Eulerian polynomials are palindromic and unimodal.  We see from (\ref{gesmac}) and (\ref{gesstan}), respectively,  that the polynomials $\sum_{w \in \mathcal {MD}_M} t^{\exc(w) } $ and  $\sum_{w \in \mathcal {SW}_M} t^{\des(w) } $ are palindromic and unimodal for all multisets $M$.
\end{remark}

\section{$q$-analog of $B_n$} \label{qbnsec} 

The lattice
$B_n(q)$ of subspaces of an $n$-dimensional vector space over the finite
field $\ff_q$ is bounded and pure of length $n$.  It is well known that
$B_n(q)$ is EL-shellable (see \cite{w1}).  Using (3.3) to compute
$\beta((B_n(q) * T_{t,n})^-)$ and equating the resulting formula with the formula given in
(\ref{gtreeintro2}), we obtain a new Mahonian permutation statistic, which we call $\aid$,
and we show that the pairs $(\aid,\des)$ and $(\maj,\exc)$ are equidistributed on
$\S_n$.

Let $\sigma \in \mathfrak S_n$.  Recall that an {\it inversion} of $\sigma$ is a pair $(\sigma(i),
\sigma(j))$ such that  $ 1 \le i < j \le n$ and $\sigma(i) > \sigma(j)$.   An {\it admissible
inversion} of $\sigma$ is an inversion $(\sigma(i),\sigma(j))$ that satisfies either
\begin{itemize} \item $1<i$ and $\sigma(i-1)<\sigma(i)$ or
\item there is some $k$ such that $i<k<j$ and $\sigma(i)<\sigma(k)$.
\end{itemize}

We write $\inv(\sigma) $ for  the number of inversions of $\sigma$ and $\rm{ai}(\sigma)$ for the number of admissible inversions of
$\sigma$.  For example, if $\sigma =  6431275$ then there are 11 inversions,  but only 
$(6,5)$ and $(7,5)$ are admissible.  So $\inv(\sigma) = 11$ and  $\rm{ai}(\sigma) = 2$.

Now let $$\aid(\sigma) := \ai(\sigma) + \des(\sigma).$$
It turns out that $\aid$ is equidistributed with the Mahonian permutation statistics $\inv$ and $\maj$ on $\S_n$.  We  give a short combinatorial proof of this in Proposition~\ref{mahonaid} below.  
First we prove the following more general joint distribution result.

\begin{thm} \label{equidistth} For all $n\ge 0$, $$ \sum_{\sigma \in \mathfrak S_n} q^{\aid(\sigma)} t^{\des(\sigma)} =  \sum_{\sigma \in \mathfrak S_n} q^{\maj(\sigma)} t^{\exc(\sigma)}.$$
\end{thm}

\begin{proof}   It is well-known (see \cite[Theorem 3.12.3]{st5}) that for all $S \subseteq [n-1]$,
$$\beta(B_n(q) _S) = \sum_{\scriptsize\begin{array}{c} \sigma \in \sg_n \\ \Des(\sigma) = S\end{array}} q^{\inv(\sigma)}.$$   Hence by  (\ref{minmaxeq2}) we have
\begin{eqnarray}\nonumber \beta((B_n(q)*T_{t,n})^-) &=& \sum_{S \in \psx([1,n-2])} \sum_{\scriptsize\begin{array}{c} \sigma \in \sg_n \\ \Des(\sigma) = [1,n-1]\setminus S \end{array}}\hspace{-.2in} q^{\inv(\sigma)}t^{|S|+1} (t+1)^{n-1-2|S| }\\ \label{qbetaeq} &=& \sum_{\scriptsize\begin{array}{c} \sigma \in \sg_n  \cap \NDA_n
\\ \sigma_{n-1} > \sigma_n \end{array}} q^{\inv(\sigma)}t^{\asc(\sigma)+1} (1+t) ^{n-1-2\asc(\sigma)}. \end{eqnarray}

We will rewrite the expression (\ref{qbetaeq}) as the enumerator of barred permutations.  Given a  set $X$ of 
size $n$, a barred permutation of $X$ is a word $w_1w_2\dots w_n$ with $n$ distinct letters in $X$,  in which 
some of the letters are barred.   Let $|w_i|$ denote the letter $w_i$ with the bar removed if there is one and let 
$|w| = |w_1|\cdots|w_n| \in \sg_X$, where $\sg_X$ is the set of ordinary permutations of $X$.   Let $\bars(w)$ 
denote the number of bars of $w$.
Let $\mathcal W_X$ be the set of  barred permutations $w$ of  $X$ satisfying
\begin{itemize}
\item[(A)] $w_n$ is barred
\item[(B)] if $i \in [n-1]$ and
$|w_i|<|w_{i+1}|$ then $w_i$ is barred and $w_{i+1}$ is not barred.
\end{itemize}
It is not hard to see that the expression (\ref{qbetaeq}) equals $$\sum_{w \in \mathcal W_{[n]}} q^{\inv(|w|)} t^{\bars(w)} ,$$ which by Lemma~\ref{bijlem} below equals $$\sum_{\sigma \in \mathcal \sg_n} q^{\binom{n}2-\ai(\sigma)} t^{\des(\sigma)+1}.$$  Hence 
\begin{equation} \label{aibeta} \beta(B_n(q)*T_{t,n}^-) =\sum_{\sigma \in \mathcal \sg_n} q^{\binom{n}2-\ai(\sigma)} t^{\des(\sigma)+1}.\end{equation}
The result now follows from (\ref{gtreeintro2}).
\end{proof}

 Given barred permutations $\alpha \in \mathcal W_A$ and $\beta \in \mathcal W_B$, where $A$ and $B$ are disjoint sets, let $\alpha \cdot \beta$ denote the barred permutation in
 $\mathcal W _{A\uplus B} $ obtained by concatenating the words $\alpha$ and $\beta$.  Also let $\theta$ denote the empty word.
    We  define a map  $$\varphi: \biguplus_{\scriptsize \begin{array}{c} X \subseteq \PP \\ |X| < \infty\end{array} } \mathcal W_X \to \biguplus_{\scriptsize \begin{array}{c} X \subseteq \PP \\ |X| < \infty\end{array} } \mathfrak S_X,$$ recursively as follows.  If $w$ is in the domain of $\varphi$  and $m$ is the maximum letter of $|w|$ then
     $$\varphi(w) = \begin{cases} \theta &\mbox{ if }w= \theta
     \\ m\cdot \varphi(\beta) &\mbox{ if  }w = \bar m\cdot \beta \\ 
     \varphi(\beta)\cdot m \cdot \varphi(\alpha) &\mbox{ if } w= \alpha\cdot  m \cdot \beta \mbox{ and }  \beta \ne \theta.  \end{cases}
  $$

 \begin{lemma} \label{bijlem} The map $\varphi$ is a well-defined bijection which satisfies
 \begin{enumerate}
 \item $\varphi(\mathcal W_X) = \mathfrak S_X$,
 \item $\des(\varphi(w ))+1= \bars(w)$
 \item $\ai(\varphi(w )) = \binom {|X|} 2 - \inv(|w|)$
  \end{enumerate} for all finite nonempty subsets $X$ of $\pp$ and all $w \in \mathcal W_X$. \end{lemma}

\begin{proof} By (B) of the definition of $\mathcal W_X$,  if letter $m $ is barred in the word $w \in \mathcal W_X$ then it is the first letter of
   $w$.  By (A), if $m$ is unbarred it cannot be the last letter.  Hence the three cases of the definition of $\varphi$ cover all possibilities.   It is  also clear from the definition of $\mathcal W_X$ that if $\alpha m \beta \in \mathcal W_X$ and $\beta \ne \theta$ then  $\alpha \in \mathcal W_A$ and $\beta \in \mathcal W_{X\setminus( {A}\cup \{m\})}$ for some subset $A\subsetneq X$.  Hence by induction on $|X|$ we have that $\varphi$ is a well-defined map that takes elements of
$\mathcal W_X$ to $\mathfrak S_X$.

To show that $\varphi$ is a bijection satisfying (1) we  construct its inverse.  Define
$$\psi: \biguplus_{\scriptsize \begin{array}{c} X \subseteq  \PP \\ |X| < \infty\end{array} } \mathfrak S_X \to \biguplus_{\scriptsize \begin{array}{c} X \subseteq \PP \\ |X| < \infty\end{array} } \mathcal W_X,$$ recursively by
  $$\psi(\sigma) = \begin{cases} \theta &\mbox{ if $ \sigma= \theta$} \\ \bar m \cdot \psi(\delta) &\mbox{ if $\sigma= m \cdot \delta$}
  \\ \psi(\delta)\cdot m\cdot  \psi(\gamma) &\mbox{ if }\sigma= \gamma \cdot m \cdot  \delta \mbox{ and }  \gamma \ne \theta \end{cases},
  $$
 where $m$ is the maximum letter of $\sigma$.
 Let $\gamma m \delta \in \sg_X$. One can see that conditions (A) and (B) of the definition of $\mathcal W_X$ hold for $\psi(\gamma m \delta)$ whenever they hold for $\psi(\gamma) $ and $\psi(\delta) $.  Hence by induction on $|X|$,  $\psi$ is a well defined map.  One can easily also show by induction that $\varphi$ and $\psi$ are inverses of each other.

We also prove (2) by induction on $|X|$, with the base case $|X|=0$ being trivial.  We do the third case of the definition of $\varphi$ and leave the second to the reader.  Let $w = \alpha m \beta \in \mathcal W_X$ with $\beta \ne \theta$.   If $\alpha \ne \theta$ then
$$\bars(w) = \bars(\alpha) + \bars(\beta) = \des(\varphi(\alpha))+ \des(\varphi(\beta))+2,$$  by the induction hypothesis. 
 Since $m$ is the largest element of $X$ and is not the last letter of $\varphi(w)$, we have $$\des(\varphi(w)) = \des(\varphi(\beta)) +1 + \des(\varphi(\alpha)).$$
Hence (2) holds in this case.

Our proof of (3) proceeds by induction
on $n=|X|$, the case $n=0$ being trivial.

 If $w  =\bar m \cdot \beta$ then
 \begin{eqnarray*}
\ai( \varphi(w)) & = & \ai(m\cdot \varphi(\beta)) \\ & = &  \ai( \varphi(\beta)) \\
& = & {{n-1} \choose {2}}-\inv(|\beta|) \\ & = & {{n} \choose
{2}}-(\inv(|\beta|)+n-1) \\ & = & {n \choose 2}-\inv(|\bar m \cdot \beta|).
\end{eqnarray*}
Indeed, the first two equalities follow immediately from the
definitions and the third follows from our inductive hypothesis.

Next, say $w=\alpha \cdot  m \cdot  \beta$ with $\alpha \in
\mathcal W_A$ and $\beta \in \mathcal W_B$, where $|B|  >0$.
Set
$\inv(A,B):=|\{(a,b):a \in A,b\in B,a>b\}$   It follows quickly
from the inductive hypothesis and the definitions that
\begin{eqnarray*}
\ai( \varphi(w))& = &\ai(  \varphi(\beta)\cdot  m\cdot \varphi(\alpha) )
\\ & = &\ai( \varphi(\beta))+ |A|+\ai( \varphi(\alpha) ) +\inv(B,A)
\\& = & {{|B|}
\choose {2}}-\inv(|\beta|)+n-1-|B|
\\ & & +{{|A|} \choose
{2}}-\inv(|\alpha|)+|A||B|-\inv(A,B).
\end{eqnarray*}
Now
\[
\inv(|\alpha\cdot m \cdot \beta|)=\inv(|\alpha|)+|B|+\inv(|\beta|) + \inv(A,B)
\]
and a straightforward calculation shows that
\[
{{|B|} \choose {2}}+n-1+{{|A|} \choose {2}}+|A||B|={{n}
\choose {2}}.
\]
Hence $$\ai( \varphi(w)) = \binom n 2 - \inv(|\alpha\cdot {m}\cdot \beta|)$$ as desired.
 \end{proof}

 We pose the question of whether there is  a nice direct bijective proof of Theorem~\ref{equidistth}.
Our proof of Theorem~\ref{equidistth} relies on (\ref{gtreeintro2}), whose proof, in turn, relies on a  $q$-analog of Euler's formula  for the Eulerian  polynomials derived by Shareshian and Wachs  in  \cite{ShWa}.  A considerable amount of work in symmetric function theory and bijective combinatorics went into the proof of this $q$-analog of Euler's formula.
Since the steps in deriving Theorem~\ref{equidistth} from the $q$-analog of Euler's formula are reversable, a nice direct combinatorial proof of    Theorem~\ref{equidistth} would provide an interesting alternative proof of the $q$-analog of Euler's formula.
    Here we   give a simple combinatorial proof  that $\aid$ is Mahonian.

\begin{prop} \label{mahonaid} Let $F_n(q) = \sum_{\sigma \in \mathfrak S_n} q^{\aid(\sigma)}$.  Then
$F_n(q)$ satisfies the following recurrence for all $n \ge 2$,
$$F_n(q) := (1+q)F_{n-1}(q) + \sum_{j=2}^{n-1}  \left[\begin{array}{c} n-1 \\j-1\end{array}\right]_q q^j F_{j-1}(q) F_{n-j} (q) .$$  Consequently $F_n(q) = [n]_q!.$
\end{prop}

\begin{proof} The terms on the right side of the recurrence $q$-count  permutations according to the position of $n$ in the permutation.  That is for each $j$, $$\sum_{\scriptsize \begin{array}{c} \sigma \in \mathfrak S_n\\ \sigma(n-j+1) = n\end{array}} q^{\aid(\sigma)} = \begin{cases} \left[\begin{array}{c} n-1 \\j-1\end{array}\right]_q  q^j F_{j-1}(q) F_{n-j} (q) &\mbox { if } j=2,\dots,n-1\\ &\\
F_{n-1}(q) &\mbox { if } j = 1 \\
q F_{n-1}(q) &\mbox { if } j=n.\end{cases}$$

It is easy to see that $[n]_q!$ also satisfies the same recurrence relation.
\end{proof}

 A  more natural Mahonian permutation statistic whose joint distribution with $\des$ is the same as that of $\aid$ is discussed in  \cite{sw3, sw4}.  This statistic is a member of a family of Mahonian statistics introduced by Rawlings \cite{ra1}.

\section{$p$-analog of chain product analog of $B_n$} \label{psec}
Given a prime  $p$ and a weak composition $\mu:=(\mu_1,\dots,\mu_k)$ of $n$, let $B_\mu(p)$ denote the lattice of subgroups of the  abelian $p$-group $\Z/ p^{\mu_1}\Z  \times \cdots \times \Z/ p^{\mu_k}\Z $.  The poset $B_\mu(p)$ is a natural $p$-analog of $B_\mu$.  It is pure and bounded of length $n$.   Moreover, it provides  the following $p$-analog of Theorem~\ref{multirees}.

 \begin{thm} \label{multireesp} Let $\mu$  be a weak composition of $n$ and let $p$ be a prime.  Then $B_\mu(p)^- * T_{t,n-1}$ and $(B_\mu(p) * T_{t,n})^-$ have the homotopy type of a wedge of $(n-1)$-spheres.  The numbers of spheres in these wedges are,
respectively,
 \begin{equation} \label{pderaneq} \beta(B_\mu(p)^- * T_{t,n-1}) = \sum_{w \in \mathcal{MD}_{M(\mu)} } p^{s_1(w)}t^{\exc(w)+1}\end{equation}
 and
\begin{equation}\label{psmireq} \beta( (B_\mu(p) * T_{t,n})^-)= \sum_{w \in \mathcal{SW}_{M(\mu)} }p^{s_2(w)} t^{\des(w)+1},
\end{equation}
where $s_1,s_2: \pp^n \to \N$ are statistics on words over $\pp$.
\end{thm}

\begin{proof}  It is well-known that $B_\mu(p)$ is EL-shellable.   Hence by Theorem~\ref{ELth}, $B_n(q)^- \ast T_{t,n-1}$  and  $(B_n(q) \ast T_{t,n})^-$ are EL-shellable. 
 It is also  known  (see \cite[(1.30)]{but}) that for all $S \subseteq [n-1]$,
$$\beta(B_{\mu}(p)_S) = \sum_{\scriptsize\begin{array}{c} w \in \sg_{M(\mu)}\\ \Des(w) =S\end{array}} p^{ \cha(w)} ,$$
where $\cha$ is a statistic on words introduced by Lascoux and Sch\"utzenberger for the purpose of showing that  the Kostka polynomials have nonnegative integer coefficients.  (We will not need  the precise definition of cocharge here.) 

Now by (\ref{minmaxeq1}) we have 
$$\beta(B_\mu(p)^-*T_{t,n-1}) = \sum_{S \in \psx([2,n-2])} \sum_{\scriptsize\begin{array}{c} w \in \sg_{M(\mu)} \\ \Des(w) = [1,n-1]\setminus S \end{array}}\hspace{-.2in} p^{\cha(w)}t^{|S|+1} (t+1)^{n-2-2|S| } .$$
Hence \begin{equation} \label{ppderangeeq} \beta(B_\mu(p)^-*T_{t,n-1}) = \sum_{i=1}^{n} f_i(p) t^{i}, \end{equation}  where $f_i(p) \in \N[p]$. By (\ref{gesmac}),  $$f_i(1)=|\{w \in \mathcal{MD}_{M(\mu)}: \exc(w)=i -1 \}|.$$  Since $f_i(1)$ is the sum of the coefficients of $f_i(p)$, we can assign a nonnegative integer $s_1(w)$ to each word $w$ in $ \mathcal{MD}_{M(\mu)}$ so that
$$f_i(p) = \sum_{\scriptsize \begin{array}{c} w \in  \mathcal{MD}_{M(\mu)} \\ \exc(w)=i -1\end{array}} p^{s_1(w)} .$$ By plugging this into  (\ref{ppderangeeq}), we obtain the desired result (\ref{pderaneq}).

The proof of (\ref{psmireq}) follows along  the lines of that  of (\ref{pderaneq}) with (\ref{minmaxeq2}) and  (\ref{gesstan})  used instead of (\ref{minmaxeq1}) and  (\ref{gesmac}).
\end{proof}

\noindent{\bf Problem:} It would be interesting to find nice combinatorial descriptions of the coefficients of the polynomials $\beta(B_\mu(p)^- * T_{t,n-1})$ and $\beta((B_\mu(p) * T_{t,n})^-)$.  That is, find  natural statistics  $s_1$ and $s_2$ for which (\ref{pderaneq}) and (\ref{psmireq}) hold.   When $w \in \sg_{M(1^n)}$, we see from (\ref{gtreeintro1}) that   $s_1(w)$ can be defined  to be $\binom n 2 -\maj(w)+\exc$  and  from (\ref{aibeta}) that $s_2(w)$ can be defined to be $\binom n 2 - \ai$.  
\vspace{.1in}

\section{The noncrossing partition lattice} \label{noncrosssec}

A set partition $\pi$ is said to be
{\em noncrossing} if for all $a<b<c<d$, whenever $a,c$ are in a block $B$ of $\pi$ and
$b,d$ are in a block $B^\prime$ of $\pi$ then $B = B^\prime$.
  Let 
$\NC_n$ be the poset of noncrossing
partitions of $[n]$ ordered by reverse refinement.  This poset, known as the {\em noncrossing partition lattice}, was
first introduced by Kreweras \cite{kr72}, who showed  that it is a pure
lattice with M\"obius invariant equal to  the signed Catalan number $(-1)^{n-1}{1 \over
n} {2n-2 \choose n-1}$.   Bj\"orner and Edelman   (cf.  \cite{bj}) gave the first EL labeling of  $\NC_n$ and later Stanley \cite{st6} gave a different EL-labeling in which the maximum chains are labeled with parking functions.    

A word $w\in \pp^n$ is said to be a {\em parking function} of length $n$  if its weakly increasing rearrangement $u$ satisfies $u_i \le i $ for all $i \in [n]$.  Let $\pf_n$ be the set of parking functions of length $n$.  Recall that for $w=w_1,\dots,w_n \in \pp^n$,  
$$\Des(w) := \{ i \in [n-1] : w_i > w_{i+1}\} \mbox{ and } \des(w) := |\Des(w)|.$$ Stanley uses his EL-labeling to prove that  for all $S \subseteq [n-1]$,
\begin{equation}\label{strank} \beta((\NC_{n+1})_{[n-1]\setminus S})= |\{ w \in \pf_n : \Des(w) = S\}|. \end{equation}

 \begin{thm} \label{noncross} For all $n,t \in \pp$, the posets   $(\NC_{n+1}  * T_{t,n})^-$ and $\NC_{n+1} ^- * T_{t,n-1}$ have the homotopy type of a wedge of $(n-1)$-spheres.  The numbers of spheres in these wedges are,
respectively,
 \begin{equation}\label{noncross2} \beta((\NC_{n+1}  * T_{t,n})^-)=  \frac {1}{n+1} \sum_{k=0}^{n-1} {n-1 \choose k} \sum_{w \in [n+1]^{n-k}} t^{\des(w)+k} .
\end{equation}
and
 \begin{equation} \label{noncross1} \beta(\NC_{n+1} ^- * T_{t,n-1}) =\end{equation} $$  (-1)^{n}+ \frac {1}{n+1}\sum_{r=0}^{n-1}(-1)^r \binom {n+1} r \sum_{k=0}^{n-1-r} {n-1-r \choose k} \sum_{w \in [n+1]^{n-k-r}} t^{\des(w)+k} . $$
\end{thm}

By a  straightforward computation  involving the binomial theorem, Theorem~\ref{noncross} reduces to the following result when $t=1$.

 \begin{cor} \label{noncrosscor}  For all $n \in \pp$, the posets  $(\NC_{n+1}  * C_{n})^-$ and $\NC_{n+1} ^- * C_{n-1}$  have the homotopy type of a wedge of $(n-1)$-spheres.   The numbers of spheres in these wedges are,
respectively,   \begin{equation*} \beta((\NC_{n+1}  * C_{n})^-)=  (n+2)^{n-1} \end{equation*}
and
 \begin{equation*}  \beta(\NC_{n+1} ^- * C_{n-1}) = \frac{(n+1)^{n+1}+(-1)^{n}(n+3)} {(n+2)^2} \end{equation*}
\end{cor}

\begin{proof}[Proof of Theorem~\ref{noncross}]  Since $\NC_{n+1}$ is EL-shellable it  follows from Theorem~\ref{ELth}\ that $\NC_{n+1} ^- * T_{t,n-1}$ and $(\NC_{n+1}  * T_{t,n})^-$ have the homotopy type of a wedge of $(n-1)$-spheres. 

{\em Proof of (\ref{noncross2})}.  By substituting (\ref{strank}) into (\ref{minmaxeq2})  we obtain

 \begin{eqnarray} \label{noncrossbetti1} & &\hspace{-.2in} \beta((\NC_{n+1}  * T_{t,n})^-) \\ \nonumber & = &\sum_{S \in \psx([1,n-2])} |\{ w \in \pf_n : \des(w) = S\}| t^{|S|+1} (t+1)^{n-2|S|-1 }\\ \nonumber
 &=& \sum_{\scriptsize\begin{array}{c}w \in \pf_n \cap  \NDD_n\\ w_{n-1} \le w_n\end{array}}t^{\des(w)+1}(t+1)^{n-2\des(w) -1} .\end{eqnarray}

 Let $\wcomp_{n,k}$ be the set of all weak  
 compositions of $n$ into $k$ parts.  It is straightforward to show that  $w\in \pf_n$ if and only if  $w \in \sg_{M(\mu)}$ for some $\mu \in \wcomp_{n,n}$ such that $\sum_{i=1}^j \mu_i \ge j$ for all $j= 1 \dots,n$.  We will call a weak composition of $n$ into $n$ parts
 that satisfies this condition a {\em parking composition} of $n$, and let $\pc_n$ be the set of all parking 
 compositions of $n$.  It now follows from (\ref{noncrossbetti1}) that
\begin{equation} \label{noncrossbetti} \beta((\NC_{n+1}  * T_{t,n})^-) = \sum_{\mu \in \pc_n} \sum_{\scriptsize\begin{array}{c}w \in \sg_{M(\mu)} \cap  \NDD_n\\ w_{n-1} \le w_n\end{array}}t^{\des(w)+1}(t+1)^{n-2\des(w) -1} .\end{equation}

 Note that  every parking composition $\mu$ of $n$ can be viewed as an element of $\wcomp_{n,n+1}$ by adjoining a $0$ to the end of  $\mu$.   For  $\mu, \mu^\prime \in \wcomp_{n,n+1}$, we say that $\mu$ and $\mu^\prime$ are cyclically equivalent if   $\mu^\prime $ can be obtained  by cyclically rotating the parts of  $\mu$.  Since all elements of $\wcomp_{n,n+1}$ are primitive words, i.e., they are not equal to a power of a shorter word, the equivalence classes of $\wcomp_{n,n+1}$ under cyclic equivalence all have size equal to $n+1$.  Moreover,  each equivalence class has exactly one  parking composition $\mu$, i.e. $\mu=(\mu_1,\dots,\mu_{n},0)$ where $(\mu_1,\dots,\mu_n) \in \pc_n$.  

Given a weak composition $\mu$ of $n$, let 
$$F_\mu:= \sum_{\scriptsize\begin{array}{c}w \in \sg_{M(\mu)} \cap  \NDD_n\\ w_{n-1} \le w_n\end{array}}t^{\des(w)+1}(t+1)^{n-2\des(w) -1}$$ and let 
$${\bf x}^{\mu} := x_1^{\mu_1} \cdots x_k^{\mu_k},$$
for  $\mu = (\mu_1,\dots, \mu_k)$.
It follows from (\ref{ges4}) that  $\sum_{\mu\in \wcomp_{n,n+1}} F_\mu {\bf x}^\mu $,  is a polynomial in $t$ whose coefficients are symmetric polynomials in the 
variables $x_1,\dots,x_{n+1}$.  Hence $$F_\mu = F_{\mu^{\prime}}$$ whenever $\mu^\prime$ is a 
rearrangement of $\mu$; so $F_\mu$ is constant on cyclic equivalence classes of $\wcomp_{n,n
+1}$.  We can therefore choose  a representative of each cyclic equivalence class of $
\wcomp_{n,n+1}$  to compute the sum of $F_\mu$ over the weak compositions $\mu$  in $
\wcomp_{n,n+1}$.  By letting the parking compositions  be the chosen representatives, we arrive 
at $$ \sum_{\mu \in \wcomp_{n,n+1}} F_\mu = (n+1) \sum_{\mu \in \pc_n}F_\mu.$$  It now follows 
from (\ref{noncrossbetti}) that 

\begin{equation}\label{bettifu}  \beta((\NC_{n+1}  * T_{t,n})^-) = \frac 1 {n+1}  \sum_{\mu \in 
\wcomp_{n,n+1}} F_\mu. \end{equation}

By combining (\ref{ges4}) and (\ref{bannereq}) we have that for all $m,n \in \pp$ and  $\mu \in \wcomp_{n,m}$,
$$F_\mu = \sum_{\scriptsize\begin{array} {c} w \in W_n\\ |w| \in  \sg_{M(\mu)} \end{array}} t^{\bars(w)},$$
which implies that
\begin{equation}\label{Fubar}  \sum_{\mu \in \wcomp_{n,m}} F_\mu = \sum_{\scriptsize\begin{array} {c} w \in W_n\\ |w| \in  [m]^n \end{array}}t^{\bars(w)}.\end{equation} .

We claim that for all $m$ and $n$,
\begin{equation}\label{bardeseq} \sum_{\scriptsize\begin{array} {c} w \in W_n\\ |w| \in  [m]^n \end{array}} t^{\bars(w)} = \sum_{k=0}^{n-1} {n-1 \choose k} t^k \sum_{u \in [m]^{n-k}} t^{\des(u)}.\end{equation}  To prove this claim  first note that there are two types of barred letters in $w \in W_n$. The type I barred letters are those that are followed by a letter equal to it in absolute value and the type II barred letters are those that are followed by a letter that is smaller than it  in absolute value.  The summand on the right side of the equation enumerates banners that have exactly $k$ barred letters of the first type.  To obtain such a banner first choose the $k$ positions from the first $n-1$ positions in which the type I  barred letters are to appear and leave them blank, then fill in the remaining $n-k$ positions with an arbitrary word $u$ in $[m]^{n-k}$, then fill in the $k$ barred positions that were left blank from right to left so that the letter equals its successor in absolute value, and finally put bars over the descent positions of the resulting word, the number of which is clearly equal to $\des(u)$.  

By combining  (\ref{Fubar}) and (\ref{bardeseq}) we obtain
\begin{equation} \label{fubarbetti} \sum_{\mu \in \wcomp_{n,m}} F_\mu = \sum_{k=0}^{n-1} {n-1 \choose k} t^k \sum_{w \in [m]^{n-k}} t^{\des(w)}.\end{equation} 
Now set $m=n+1$ and plug this equation into (\ref{bettifu}) to obtain the desired result (\ref{noncross2}).

{\em Proof of (\ref{noncross1})}.  It follows from (\ref{strank}) and (\ref{minmaxeq1}) that 
$$\beta(\NC_{n+1}^-*T_{t,n-1}) = \sum_{\mu \in \pc_n} G_\mu ,$$ where 
$$G_\mu:= \sum_{\scriptsize\begin{array}{c}w \in \sg_{M(\mu)} \cap  \NDD_n\\ w_{n-1} \le w_n \\ w_1 \le w_2 \end{array}}t^{\des(w)+1}(t+1)^{n-2\des(w) -2}.$$
Using a similar argument to that which was used to derive (\ref{bettifu})  (with (\ref{ges3}) now playing the role of (\ref{ges4})) we obtain 
\begin{equation} \label{betag} \beta(\NC_{n+1}^-*T_{t,n-1}) = \frac 1 {n+1} \sum_{\mu \in \wcomp_{n,n+1}} G_\mu.\end{equation}

For $n \ge 1$, let
$$F_n:= \sum_{\scriptsize\begin{array}{c} w \in \NDD_n\\ w_{n-1} \le w_n \end{array}}  t^{\des(w)} (1+t) ^{n-1-2\des(w)} {\bf x}_w$$
and 
$$G_n:= \sum_{\scriptsize\begin{array}{c} w \in \NDD_n\\ w_1 \le w_2\\ w_{n-1} \le w_n \end{array}}  t^{\des(w)} (1+t) ^{n-2-2\des(w)} {\bf x}_w.$$
Also let $F_0 = G_0 =1$.  By (\ref{ges4}) and (\ref{ges3}) we have
$$\sum_{n \ge 0 }F_n z^n = \sum_{n \ge 0} G_n z^n  \sum_{n\ge 0} h_n z^n.$$  Hence
$$ \sum_{n \ge 0} G_n z^n = \sum_{n \ge 0 }F_n z^n  \sum_{n\ge 0}(-1)^i e_n z^n.$$
Equating coefficients of $z^n$ yields
\begin{equation} \label{gneq} G_n = \sum_{r=0}^n (-1)^r e_r F_{n-r}.\end{equation}
By applying to (\ref{gneq}), the specialization that sets
$$x_i = \begin{cases} 1 & \mbox{ if } i \in [n+1] \\ 0 &\mbox { otherwise}, \end{cases}$$
we obtain

$$ \sum_{\mu \in \wcomp_{n,n+1}} G_\mu = (-1)^n (n+1) + \sum_{r=0}^{n-1}(-1)^r \binom {n+1} r  \sum_{\mu \in \wcomp_{n-r,n+1}} F_\mu.$$  By plugging (\ref{fubarbetti}) into this equation we obtain 
\begin{eqnarray} \nonumber  \sum_{\mu \in \wcomp_{n,n+1}} G_\mu  &=& (-1)^n (n+1) + \\  \nonumber & &  \hspace{-.4in}  \sum_{r=0}^{n-1}(-1)^r \binom {n+1} r \sum_{k=0}^{n-r-1} \binom {n-r-1}k  t^k \sum_{w \in [n+1]^{n-r-k}}t^{\des(w)}.\end{eqnarray}
The desired result (\ref{noncross1})  follows from this and (\ref{betag}).
 \end{proof}

\end{document}